\documentclass[11pt,reqno]{amsart}

\usepackage[T1]{fontenc}
\usepackage{graphicx}

\usepackage{color}
\definecolor{MyLinkColor}{rgb}{0,0,0.4}

\newcommand{\R}{{\mathbb R}}
\newcommand{\Z}{{\mathbb Z}}
\newcommand{\N}{{\mathbb N}}
\newcommand{\s}{\mathbb S}

\newcommand{\ov}{\overline}
\newcommand{\p}{\partial}
\newcommand{\0}{\Omega}

\newcommand{\e}{\varepsilon}

\newtheorem{thm}{Theorem}[section]
\newtheorem{prop}[thm]{Proposition}
\newtheorem{lemma}[thm]{Lemma}
\newtheorem{cor}[thm]{Corollary}

\theoremstyle{remark} 
\newtheorem{rem}[thm]{Remark}

\numberwithin{equation}{section}
\title[On the analyticity of periodic gravity water waves ]{On the analyticity of periodic gravity water waves with integrable vorticity function}


\author[J. Escher]{Joachim Escher}
\address{Institut f{\"u}r Angewandte Mathematik, Leibniz Universit{\"a}t Hannover, Welfengarten~1, 30167 Hannover, Germany.}
\email{escher@ifam.uni-hannover.de}

\author[B.--V. Matioc]{Bogdan--Vasile Matioc}
\address{Institut f{\"u}r  Mathematik,  Universit{\"a}t Wien, Nordbergstrasse 15, 1090 Wien, \"Osterreich. }
\email{bogdan-vasile.matioc@univie.ac.at}

\usepackage[colorlinks=true,linkcolor=MyLinkColor,citecolor=MyLinkColor]{hyperref} 

\begin{document}

\begin{abstract}
We prove that the streamlines and the profile of periodic gravity water waves 
 traveling  over a flat bed with wavespeed which exceeds the horizontal velocity of all fluid particles are real-analytic graphs if the  vorticity function is merely integrable.
\end{abstract}

\maketitle

\section{Introduction and the main result}

This paper is dedicated to the study of the regularity properties of two-dimensional periodic gravity water waves traveling over a flat bed when gravity is the sole driving mechanism.
The waves we consider are rotational and the vorticity function has a general form. 
While there are no known non-trivial solutions that describe periodic gravity water waves traveling over a flat bed,  in the case when the fluid has infinite depth there is an explicit solution  which is due to Gerstner \cite{Co01, Ge09, He08, AM12a}.
This solution has a  non-trivial vorticity function, the streamlines being real-analytic   trochoids. 
It is interesting that Gerstner's example   is the only possible non-trivial solution for gravity water waves without stagnation points and with the pressure constant along all streamlines, cf. \cite{MM12}. 
Small amplitude gravity water waves possessing a general vorticity were constructed first in \cite{DJ34} by using power series expansions,   the existence of waves of large amplitude and without stagnation points
being established more recently \cite{CoSt04} by using  local and  global bifurcation techniques.
While   in \cite{CoSt04} the vorticity function  is assumed to have a H\"older continuous derivative, the same authors construct  in \cite{CS11} gravity water waves with a vorticity function which is merely bounded.
This new  result is obtained by taking advantage of the weak formulation for the water wave problem.  
In the irrotational setting, when the vorticity is zero and the waves travel  over still water or uniformly sheared  currents, the existence theory  is related to   Nekrasov's equation.
Based on this formulation of the problem,   irrotational water waves with small and large amplitude were constructed 
in \cite{AFT82, KN78,  To96} by using  global bifurcation theory, some of the waves possessing  sharp crests with a stagnation point at the top.

When proving the real-analyticity of gravity water waves, the assumption that the flow does not possess stagnation points, that is fluid particles that travel horizontally with the same speed as the wave, is crucial.
Indeed, as a consequence of elliptic maximum principles, irrotational waves cannot possess stagnation points beneath the wave profile (see e.g. \cite{CoSt10}), but the Stokes wave of extreme form has a stagnation point at each crest and the wave profile is 
only Lipschitz continuous in a neighbourhood of that point as it forms an angle of $2\pi/3$.
On the other hand, it is known that waves with vorticity may possess stagnation points inside the fluid layer, see \cite{EEW11, W09} for the case when the vorticity function is constant or linear, 
 and that the streamlines containing the stagnation points are not real-analytic. 
For  irrotational  waves, it is shown in \cite{L52} by using the  Schwartz reflection principle that in the absence of stagnation points the profile of the waves is real-analytic. 
In the context of rotational waves without stagnation points, it was proved first for waves with  H\"older continuous vorticity \cite{AC11} and later for waves with merely bounded vorticity \cite{BM11}
that all the streamlines beneath the wave surface are real-analytic. 
The authors of \cite{AC11} establish the real-analyticity of the wave profile only under the additional assumption that the vorticity function is real-analytic.
The methods  rely strongly on  invariance of the problem with respect to horizontal translations and they
have been generalized to prove similar results for capillary and capillary-gravity waves \cite{Hen10, DH11a, DH12}, deep-water waves \cite{Mat12},
solitary waves \cite{H11, MM12-a},  and stratified waves \cite{HM} (see also the  survey \cite{Esch-reg12}).
While in these references the real-analyticity of the wave profile  is established only for real-analytic vorticity functions, it is shown \cite{LW12x, BM13, W13}
that any of the waves mentioned before has a real-analytic profile if the vorticity function 
is merely H\"older continuous.  
In the later papers the authors introduce additionally an  iteration procedure and  estimate all partial  derivatives of an associated height function with respect to the horizontal variable in order to obtain the desired regularity result.
We enhance that capillary and capillary-gravity waves with a real-analytic and decreasing vorticity  function have a real-analytic profile, cf. \cite{Ma12a}, the  result being true for waves with or without stagnation points. 
The smoothness of the free surface for  gravity waves without stagnation points on the surface  and for    waves with capillary effects has been established in  \cite{CV11, CM12x-1, CM12x}  
 in the regime when the  vorticity  is constant but the profile is not necessarily a graph (waves with overhanging profiles). 
 
 We establish herein the real-analyticity of the free surface and of the streamlines for solutions of the weak formulation of the water wave problem derived in \cite{CS11}, when assuming only integrability of the vorticity  function.  
 Such solutions describe water waves traveling over currents which present sudden changes with respect to the depth in form of a  discontinuous vorticity  or waves generated by wind that possess a thin layer of high vorticity at the surface
 (see \cite{Ks08, PB74}). 
 We enhance that the regularity results obtained in \cite{AC11,  LW12x, BM11} appear as a particular case of our analysis.
 The relevance of our result can be viewed in the light of  \cite[Theorem  3.1  and Remark 3.2]{MM13} as we can now state: within the set of all periodic gravity waves without stagnation points
 the symmetric waves with one crest and trough per period 
are  characterized by the property that all the streamlines have a global minimum on the same vertical line.
Another characterization for  symmetric gravity water waves with one crest and trough per period was obtained in \cite{CoEhWa07,CoEs04_b} where it is shown that merely the fact that the wave profile has a unique crest per period ensures the symmetry of the wave.  
The proof  our main result Theorem \ref{MT} combines the invariance of the problem with respect to horizontal translations  with Schauder estimates for weak solutions of elliptic boundary value problems.
Particularly, we show that even under this weak integrability condition on the vorticity function, all  derivatives of the  height function with respect to the horizontal variable have H\"older continuous derivatives and are weak solutions of certain 
elliptic problems.
Estimating their H\"older norm, we obtain  the desired regularity result.

To  complete this section  we present the governing equations  and our main result Theorem \ref{MT}.
We consider herein the  water wave problem in the formulation for the height function $h$    
\begin{equation}\label{H}
\left\{
\begin{array}{rllll}
(1+h_q^2)h_{pp}-2h_ph_qh_{pq}+h_p^2h_{qq}-\gamma(p)h_p^3&=&0&\text{in}&\0,\\[1ex]
1+h_q^2+(2gh-Q)h_p^2&=&0&\text{on} &p=0,\\[1ex]
h&=&0&\text{on}&p=p_0,
\end{array}
\right.
\end{equation}
where $\0:=\s\times(p_0,0)$ and the function $h$ is assumed to satisfy additionally
\begin{equation}\label{CH}
\inf_\0 h_p >0.
\end{equation}
We denoted by $\gamma$  the vorticity function,   $g$ is  the gravity constant, $p_0<0$ is the relative mass flux, and $Q$ is the  total head.
We set $\s:=\R/(2\pi\Z)$ to denote the unit circle.
The equivalence of the height function formulation \eqref{H}-\eqref{CH} to the Euler equations is discussed in detail in \cite{Con11, CoSt04} in the context of smooth solutions and in \cite{CS11, Esch-reg12}
for $L_r$-solutions\footnote{Given $1\le r\le\infty$, we denote the usual Lebesgue spaces by $L_r$.}, see also the remarks subsequent to Theorem \ref{MT} below.  

Solutions of \eqref{H}-\eqref{CH} describe two-dimensional $2\pi-$periodic gravity water waves traveling over the flat bed $y=-d,$ the wave profile being the graph $y=h(q,0)-d.$
Hereby, $d$ is an arbitrary real constant.
The value $h(q,p)$ represents  the exact  height  of the water particle   determined by the coordinate $(q,p)$  above the horizontal bed.
Because of \eqref{CH}, these solutions correspond to waves without stagnation points and critical layers, each streamline being a graph $y=h(q,p)-d$, for some unique $p\in[p_0,0].$ 
Indeed, following the wave from a frame moving with the wavespeed $c$, which does not appear in \eqref{H} as a solution of \eqref{H} solves the water wave problem for any value of $c$, cf. \cite{CoSt04},
the velocity field $(u,v)$ of the fluid is given by
\[
(u-c,v)=\left(-\frac{1}{h_p}, -\frac{h_q}{h_p}\right),
\]
and it follows readily from \eqref{CH} that  the horizontal speed of each individual particle is less than the wave speed $c$.  
Moreover, since the streamlines are curves in the fluid which are tangent to the velocity field, 
it is easy to see  that   these are the graphs $y=h(q,p)-d$, $p\in[p_0,0],$ as the tangent to each graph is always parallel to the velocity field and, by letting $p$ vary between $p_0$ and $0$, the graphs foliate the fluid domain. 

Assuming only boundedness and a sign condition on  the vorticity function $\gamma,$ the authors establish in \cite{CS11} the existence of weak solutions of \eqref{H}-\eqref{CH} which form a $C^1$-bifurcation curve.
The local branch can be continued if one assumes that the vorticity is H\"older  continuous close to the free surface and the bed, a characterization of the global branch  being also included.
These weak solutions belong to the space $C^{1+\alpha}(\ov\0)$ for some $\alpha\in(0,1)$ and are solutions of \eqref{H}-\eqref{CH} in the sense that they satisfy  the last two  equations of \eqref{H}  and the condition \eqref{CH} pointwise,
while the first equation is satisfied in a weak sense. 
More precisely, introducing the anti-derivative
\[\Gamma(p):=\int_0^p\gamma(s)\, ds \qquad\text{for $p\in(p_0,0)$,}\]
the first equation of \eqref{H} may be written in the weak form  
\begin{equation}\label{WH}
\left(\frac{h_q}{h_p}\right)_q-\left(\Gamma(p)+\frac{1+h_q^2}{2h_p^2}\right)_p=0\qquad\text{in $\0,$}
\end{equation}
and the weak solutions found in \cite{CS11} satisfy \eqref{WH} when testing with functions from $C^1(\ov\0) $ that have compact support.
In fact, having a bounded vorticity function, the authors of  \cite{CS11}  prove that their weak solutions belong to $W^2_r(\0)$, with $r:=2/(1-\alpha)$.
We come now the the main result of the paper.

\begin{thm}\label{MT}
 Assume that $\gamma\in L_1((p_0,0))$. 
Given a weak solution $h\in C^{1+\alpha}(\ov\0)$ of \eqref{H}, we have that $\p_q^m h\in C^{1+\alpha}(\ov\0)$ for all $m\in\N.$ 
Moreover, there exists a constant $L>1$
 with the property that
 \begin{equation}\label{E}
  \|\p_q^m h\|_{{1+\alpha}}\leq L^{m-2}(m-3)!
 \end{equation}
for all integers $m\geq3.$
\end{thm}

We enhance that if $\gamma\in L_r((p_0,0))$ for some $r>2$ then any solution $h\in W^2_r(\0)$ of \eqref{H}-\eqref{CH} is a weak solution\footnote{Recall that Sobolev's embedding theorem ensures that
$W_r^2(\Omega)\subset C^{1+\beta}(\overline{\Omega})$ for any $\beta\in (0,(r-2)/r)$.} of the same problem in the sense of \eqref{WH}.
Particularly, our result applies for all  $L_r$-solutions of \eqref{H}-\eqref{CH} found in \cite{CS11} and studied also in \cite{Esch-reg12}, improving  
the previous regularity results \cite{ AC11, LW12x, BM11}.
Furthermore, consider the classical water wave problem for a $2\pi$-periodic traveling gravity water wave with profile $\eta$, i.e.
\begin{equation}\label{eq:2eaa}
\left.
\begin{array}{rllll}
(u-c) u_x+vu_y&=&-P_x&\text{in}&\Omega_\eta,\\[0.1cm]
(u-c) v_x+vv_y&=&-P_y-g&\text{in}&\Omega_\eta,\\[0.1cm]
u_x+v_y&=&0&\text{in}&\Omega_\eta,\\[0.1cm]
P&=&0&\text{on}& y=\eta,\\[0.1cm]
v&=&(u-c)\eta_x&\text{on}&y=\eta,\\[0.2cm]
v&=&0&\text{on}& y=-d,
\end{array}
\right\}
\end{equation}
where we denote by
$
\Omega_\eta:=\{(x,y)\,;\,x\in\mathbb{S}\ \text{and}\ -d<y<\eta(x)\}
$
the ($2\pi$-periodic) fluid body beneath the wave $\eta$, by $P$ the pressure in $\Omega_\eta$, and by $g$ the gravitational constant.
Assume that
$$
(u,\,v,\,P,\,\eta)\in
W_r^1(\Omega_\eta)\times W_r^1(\Omega_\eta)\times W_r^1(\Omega_\eta)\times W_{r}^{2-1/r}(\mathbb{S})
$$ 
satisfies \eqref{eq:2eaa} in $L_r(\Omega_\eta)$. 
If we assume additionally that $u < c $ in $\overline{\Omega}_\eta$, then it is shown in  
\cite{Esch-reg12} that the height function $h$ is well-defined, belongs to $W^2_r(\Omega)$, satisfies \eqref{CH}, and is a $L_r$-solution of \eqref{H}. 
Besides, the vorticity function $\gamma$ is also well-defined and belongs to $L_r((p_0,0))$.
Thus, Theorem \ref{MT} is applicable for such solutions of \eqref{eq:2eaa}.

As a direct consequence of   Theorem \ref{MT} we obtain that the streamlines of the wave corresponding to a weak solution $h$ of \eqref{MT}-\eqref{CH}  are graphs of real-analytic functions.
Indeed, for any fixed $p\in[p_0,0],$ the function $h(\cdot,p)$ is a smooth function on $\s.$
Moreover, in virtue of \eqref{E},  we can use the Lagrange formula for the remainder to obtain that
\begin{align*}
 \left|h(q,p)-\sum_{m=0}^N \frac{\p_q^m h(q_0,p)}{m!}(q-q_0)^m\right|\leq \frac{\|\p_q^{N+1}h\|_0}{(N+1)!}|q-q_0|^{N+1}\leq L^{-2}(L|q-q_0|)^{N+1}\to_{N\to\infty} 0
\end{align*}
if  $L|q-q_0|<1.$
Thus, the Taylor series of $h(\cdot,p)$ at $q_0$  converges  on a small interval containing $q_0$, the length of the interval being independent of $q_0$ and $p$. 
Summarizing, we have:
\begin{cor}\label{COR}
Under the assumptions of Theorem \ref{MT}, given $p\in [p_0,0]$, each streamline $\Psi_p:=\big\{\big(q,h(q,p)-d\big)\in \mathbb{R}^2\,;\, q\in \mathbb{S}\big\}$ 
is a real-analytic curve. 
Particularly, the water wave's free surface $\Psi_0$ is a real-analytic periodic curve.
\end{cor}

\begin{rem} In the proof  of Theorem \ref{MT} we did not make use of the periodicity of $h$ in the variable $q$.
Therefore, the argument preceding  Corollary \ref{COR} yields also the real-analyticity of the streamlines 
and of the wave profile of gravity  solitary waves of finite depth.
\end{rem}

We denote in this paper by  $C_i,$ $i\in\N,$  universal constants which are independent of $m$ and $h$,
while $K$, $K_i,$ $i\in\N,$ denote  constants which may depend on  $\|\p_q^l h\|_{1+\alpha}$ with $0\leq l\leq 2 $ but are independent of $m$.

\section{The proof of the main result}

As a first partial result, we prove the following proposition showing that  $\p_q^m h\in C^{1+\alpha}(\ov \0)$  for all $m\in\N$ with $m\geq1$ and that $\p_q^m h$  is a weak  solution (see \cite{CS11} and Chapter 8 of \cite{GT01})
of an elliptic boundary problem.
Partially, the statement can be found in \cite{CS11}, but we are more precise and determine the coefficients, which depend on the lower $q-$derivatives of $h$, of the equations solved by $\p_q^m h$.
The exact  expressions of these coefficients are very important when proving  Theorem \ref{MT}.
\begin{prop}\label{P:1}
Given $m\in\N $ with $m\geq1,$ the function $\p_q^m h$ belongs to $ C^{1+\alpha}(\ov \0)$ and it is a weak solution of the elliptic boundary value problem
 \begin{equation}\label{WH2}
\left\{
\begin{array}{rllll}
\left(\frac{1}{h_p}\p_q w\right)_q-\left(\frac{h_q}{h_p^2}\p_p w\right)_q-\left(\frac{h_q}{h_p^2}\p_q w\right)_p+\left(\frac{1+h^2_q}{h_p^3}\p_p w\right)_p&=&\left( f_m\right)_q+\left(g_m\right)_p&\text{in}&\0,\\[2ex]
h_qw_q+(2gh-Q)h_pw_p&=&\varphi_m&\text{on} &p=0,\\[2ex]
w&=&0&\text{on}&p=p_0,
\end{array}
\right.
\end{equation}
\end{prop}
whereby $f_m, g_m, \varphi_m\in C^\alpha(\ov\0)$   are given by 
\begin{align*}
f_m:=&\sum_{n=1}^{m-1}\begin{pmatrix}m-1\\n\end{pmatrix}\left[-\p_q^n\left(\frac{1}{h_p}\right)\p_q(\p_q^{m-n}h) +\p_q^n\left(\frac{h_q}{h_p^2}\right)\p_p(\p_q^{m-n}h)\right],\\
g_m:=&\sum_{n=1}^{m-1}\begin{pmatrix}m-1\\n\end{pmatrix}\left[ \p_q^n\left(\frac{h_q}{h_p^2}\right)\p_q(\p_q^{m-n}h) -\p_q^n\left(\frac{1+h^2_q}{h_p^3}\right)\p_p(\p_q^{m-n}h)\right],
\end{align*}
and
\begin{align*}
\varphi_m:=&-2^{-1}(2gh-Q)\sum_{n=1}^{m-1} \begin{pmatrix}m\\n\end{pmatrix}(\p_q^nh_p)(\p_q^{m-n}h_p)-g\sum_{n=1}^m\begin{pmatrix}m\\n\end{pmatrix}(\p_q^nh)(\p_q^{m-n}h_p^2).
\end{align*}
\begin{proof}
For every $\e\in(0,1),$ let $\tau_\e f$ denote the translation in the $q$-direction of a given function $f\,:\,{\overline{\Omega}}\to \mathbb{R}$, i.e. $\tau_\e f(q,p):=f(q+\e,p)$ for $(q,p)\in{\overline{\Omega}}$. 
Further we define the difference quotient $u^\e:=(\tau_\e h-h)/\e\in C^{1+\alpha}(\ov\0)$ and, because $\tau_\e h$ is also a weak solution of \eqref{H}, we obtain that $u^\e$  solves the following generalized elliptic  problem
\begin{equation}\label{WH3}
\left\{
\begin{array}{rllll}
\left(a_{11}^\e\p_q u^\e\right)_q+\left(a_{12}^\e\p_p u^\e\right)_q
+\left(a_{21}^\e\p_q u^\e\right)_p+\left(a_{22}^\e\p_p u^\e\right)_p&=&0&\text{in}&\0,\\[1ex]
(h_q+(\tau_\e h)_q) u_q^\e+(2g(\tau_\e h)_p-Q)(h_p+(\tau_\e h)_p)u_p^\e&=&-2gh_p^2u^\e&\text{on} &p=0,\\[1ex]
u^\e&=&0&\text{on}&p=p_0,
\end{array}
\right.
\end{equation}
whereby 
\begin{align*}
 (a_{ij}^\e)=
 \begin{pmatrix}
  \displaystyle \frac{1}{(\tau_\e h)_p}& \displaystyle -\frac{h_q}{h_p(\tau_\e h)_p}\\[2ex]
  \displaystyle -\frac{h_q+(\tau_\e h)_q}{ 2(\tau_\e h)_p^2} & \displaystyle  \frac{(h_p+(\tau_\e h)_p)(1+h_q^2)}{2h^2_p(\tau_\e h)^2_p}
 \end{pmatrix}.
\end{align*}
 If $\e\in(0,\e_0)$ and  $\e_0$ is sufficiently small, then the matrix $(a_{ij}^\e)$ has positive eigenvalues bounded away from zero  uniformly in $\0$ and $\e\in(0,\e_0)$, while the boundary condition of \eqref{WH3} on $p=0$ is  uniformly oblique as 
 \begin{equation}\label{OC}
 \inf_{\e\in(0,1)}\inf_\0 |(2g(\tau_\e h)_p-Q)(h_p+(\tau_\e h)_p)|>0.
 \end{equation}
 Moreover,   all the coefficients appearing in \eqref{WH3}, as well the right-hand side of the second equation of \eqref{WH3}, are uniformly bounded, with respect to $\e\in(0,\e_0),$ in $C^{\alpha}(\ov\0).$
 Schauder estimates for elliptic problems, cf. \cite[Theorem 3]{CS11}, guarantee  now the existence of a constant $K>0$ such that
 \begin{equation}\label{EQQ}
 \|u^\e\|_{ 1+\alpha }\leq K\|u^\e\|_{\alpha}\leq K\|h\|_{1+\alpha}
 \end{equation}
 for all $\e\in(0,\e_0)$.
 Consequently, $(u^\e)_\e$ is bounded in $C^{1+\alpha}(\ov\0)$ and a subsequence of it converges in $C^{1}(\ov\0)$ towards $h_q$.
 In view of the estimate \eqref{EQQ}, it follows  that $h_q\in C^{1+\alpha}(\ov\0).$ 
 Letting $\e\to0$ in \eqref{WH3}, we find due to the convergence in $C^{1}(\ov\0)$ that $\p_qh$ solves \eqref{WH2} when $m=1$.

 For the general case, let us assume that $\p_q^n h\in C^{1+\alpha}(\ov \0)$ for all $1\leq n\leq m-1.$
 We are left to prove that $\p_q^m h\in C^{1+\alpha}(\ov\0)$ is the solution of \eqref{WH2}.
 Similarly as before, we define the quotient $w^\e:=\left[\tau_\e(\p_q^{m-1} h)-\p_q^{m-1}h\right]/\e\in C^{1+\alpha}(\ov\0)$, and we observe that it solves the following problem
 \begin{equation}\label{WH4}
\left\{
\begin{array}{rllll}
 \left(\frac{1}{h_p}\p_q w^\e\right)_q-\left(\frac{h_q}{h_p^2}\p_p w^\e\right)_q-\left(\frac{h_q}{h_p^2}\p_q w^\e\right)_p+\left(\frac{1+h^2_q}{h_p^3}\p_p w^\e\right)_p&=&(f_m^\e)_q+(g_m^\e)_p&\text{in}&\0,\\[2ex]
(h_q+(\tau_\e h)_q) w_q^\e+(2g(\tau_\e h)_p-Q)(h_p+(\tau_\e h)_p)w_p^\e&=&\varphi_m^\e&\text{on} &p=0,\\[2ex]
u^\e&=&0&\text{on}&p=p_0,
\end{array}
\right.
\end{equation}
with
\begin{align*}
f_m^\e:=&\p_q^m (\tau_\e h)\frac{(\tau_\e h)_p-h_p}{\e h_p(\tau_\e h)_p}-\frac{\p_p\p_q^{m-1}(\tau_\e h)}{\e}\left(\frac{h_q}{h_p^2}-\frac{(\tau_\e h)_q}{(\tau_\e h)_p^2}\right)+\frac{\tau_\e f_{m-1}-f_{m-1}}{\e},\\
g_m^\e:=&\frac{ \p_p\p_q^{m-1 }(\tau_\e h)}{\e}\left(\frac{1+h_q^2}{h_p^2}-\frac{1+(\tau_\e h)_q^2}{(\tau_\e h)_p^2}\right)
-\frac{ \p_q^{m }(\tau_\e h)}{\e}\left(\frac{h_q}{h_p^2}-\frac{(\tau_\e h)_q}{(\tau_\e h)_p^2}\right)+\frac{\tau_\e g_{m-1}-g_{m-1}}{\e},\\
\varphi_m^\e:=&\p_p \p_q^{m-1}(\tau_\e h)\frac{(2gh-Q)h_p-(2g\tau_\e h-Q)(\tau_\e h)_p}{\e}-\p_q^{m}(\tau_\e h)\left(\frac{(\tau_\e h)_q-h_q}{\e}\right)\\
&+\frac{\tau_\e \varphi_{m-1}-\varphi_{m-1}}{\e}.
\end{align*}
We note that the induction assumption  ensures  that all coefficients  and the terms on the right-hand side of the equations of  \eqref{WH4} are bounded in  $C^{\alpha}(\ov\0),$ uniformly in $\e\in(0,1).$
In virtue of \eqref{OC} and noticing that the matrix $(a_{ij}^0)$ has positive eigenvalues bounded away from zero  uniformly in $\0$, we may apply the Schauder estimate  \cite[Theorem 3]{CS11}  to problem \eqref{WH4} 
to find  a positive constant $M$  such that
\begin{equation*}
 \|w^\e\|_{C^{1+\alpha}(\ov\0)}\leq M 
 \end{equation*}
 for all $\e\in(0,1).$
Consequently, a subsequence of of $(w^\e)_\e$ converges in $C^{1}(\ov\0)$ towards $\p_q^mh$.
The same arguments as above yield that  $\p_q^mh \in C^{1+\alpha}(\ov\0)$ and, letting $\e\to0$ in the equations of \eqref{WH4}, we find that $\p_q^mh $ is the weak solution of \eqref{WH2}.
\end{proof}

Before proving the main theorem, we need the following auxiliary results.
\begin{lemma}\label{L:1}  Let $N\geq 3$ and assume that $\p_q^n u_i\in C^\alpha(\ov \0)$ for all $0\leq n\leq N$ and $1\leq i\leq 5.$
\begin{itemize}
\item[$(i)$] If   $L\geq 1$ and $\|\p_q^n u_i\|_\alpha\leq L^{n-3/2}(n-2)! $
for all $2\leq n\leq N,$
then there exists a constant $C_0>1$ with the property that
\begin{equation}\label{E:1}
 \|\p_q^n(u_1u_2u_3u_4u_5)\|_\alpha\leq C_0\left(1+\sum_{i=1}^5\sum_{l=0}^1\|\p_q^l u_i\|_\alpha\right)^{16}L^{n-3/2}(n-2)!\qquad\text{for all $2\leq n\leq N.$}
\end{equation}
\item[$(ii)$] If   $L\geq 1$ and $\|\p_q^n u_i\|_\alpha\leq L^{n-1}(n-2)! $
for all $2\leq n\leq N,$ then there exists a constant $C_1>1$ such  that
\begin{equation}\label{E:2}
 \|\p_q^n(u_1u_2u_3)\|_\alpha\leq C_1\left(1+\sum_{i=1}^3\sum_{l=0}^1\|\p_q^l u_i\|_\alpha\right)^6 L^{n-1}(n-2)!\qquad\text{for all $2\leq n\leq N.$}
\end{equation}
\item[$(iii)$] If   $L\geq 1$ and $\|\p_q^n u_i\|_\alpha\leq L^{n-2}(n-3)! $
for all $3\leq n\leq N,$ then there exists a constant $C_2>1$ with the property that
\begin{equation}\label{E:9}
 \|\p_q^n(u_1u_2)\|_\alpha\leq C_2\left(1+\sum_{i=1}^2\sum_{l=0}^2\|\p_q^l u_i\|_\alpha\right)^2 L^{n-2}(n-3)!\qquad\text{for all $3\leq n\leq N.$}
\end{equation}
 \end{itemize}
 The constants $C_0$, $C_1,$ and $C_2$ do not depend on $L$.
\end{lemma}
\begin{proof} We prove only $(i)$, the proof of $(ii)$  and $(iii)$ being similar.
Given $2\leq n\leq N,$  we have
\begin{equation}\label{MS}
 \p_q^n(u_1u_2)=\left(\sum_{k=0}^1+\sum_{k=2}^{n-2} +\sum_{k=n-1}^{n}\right)\begin{pmatrix}n\\k\end{pmatrix}(\p_q^{k}u_1)\p_q^{n-k}u_2,
 \end{equation}
 and, since $\|u_1u_2\|_\alpha\leq \|u_1\|_\alpha\|u_2\|_\alpha$, we find 
 \begin{align*}
  \left\|\left(\sum_{k=0}^1+\sum_{k=n-1}^{n}\right)\begin{pmatrix}n\\k\end{pmatrix}(\p_q^{k}u_1)\p_q^{n-k}u_2\right\|_\alpha\leq2\left(1+\sum_{i=1}^2\sum_{l=0}^1\|\p_q^l u_i\|_\alpha\right)^2L^{n-3/2}(n-2)!.
 \end{align*}
 On the other hand, if $n\geq4 $  the middle sum in \eqref{MS} does not vanish, and we use the convergence of the series $\sum_n n^{-2}$ to find that
\begin{align*} 
  \left\|\sum_{k=2}^{n-2}\begin{pmatrix}n\\k\end{pmatrix}(\p_q^{k}u_1)\p_q^{n-k}u_2\right\|_\alpha&\leq \sum_{k=2}^{n-2}\frac{n!}{k!(n-k)!} L^{k-3/2}(k-2)!L^{n-k-3/2}(n-k-2)!\\
  &\leq L^{n-3} (n-2)!\sum_{k=2}^{n-2}\frac{n^2}{(k-1)^2(n-k-1)^2}=CL^{n-3}(n-2)!.
 \end{align*}
 Summarizing, we have shown that there exists a constant $C>1$ with 
 \begin{equation}\label{E:3}
 \|\p_q^n(u_1u_2)\|_\alpha\leq C\left(1+\sum_{i=1}^2\sum_{l=0}^1\|\p_q^l u_i\|_\alpha\right)^2L^{n-3/2}(n-2)!\qquad\text{for all $2\leq n\leq N.$}
\end{equation}
Applying the estimate \eqref{E:3} to the functions
\begin{align*}
 v_1:=\frac{u_1u_2}{C\left(1+\sum_{i=1}^2\sum_{l=0}^1\|\p_q^l u_i\|_\alpha\right)^2} \qquad\text{and}\qquad v_2:=\frac{u_3u_4}{C\left(1+\sum_{i=3}^4\sum_{l=0}^1\|\p_q^l u_i\|_\alpha\right)^2},
\end{align*}
 which verify $\|\p_q^n v_i\|_\alpha\leq L^{n-3/2}(n-2)! $ for all $2\leq n\leq N,$ we get that
  \begin{equation*}
 \|\p_q^n(u_1u_2u_3u_4)\|_\alpha\leq C\left(1+\sum_{i=1}^4\sum_{l=0}^1\|\p_q^l u_i\|_\alpha\right)^8L^{n-3/2}(n-2)!
\end{equation*}
for all $3\leq n\leq N.$
Finally, we use the estimate \eqref{E:3} for the functions
\begin{align*}
u_5\qquad\text{and}\qquad u_6:=\frac{u_1u_2u_3u_4}{C\left(1+\sum_{i=1}^4\sum_{l=0}^1\|\p_q^l u_i\|_\alpha\right)^8}
\end{align*}
and obtain the desired estimate \eqref{E:1}.
\end{proof}

We use now Lemma \ref{L:1} to prove the following estimate. 
\begin{lemma}\label{L:2}  Assume that $\p_q^n u\in C^\alpha(\ov \0)$ for all $0\leq n\leq N$, with $N\geq 3,$ and let $C_1$ be the constant determined in Lemma \ref{L:1} $(ii)$.
If  there exists a constant  
\begin{equation}\label{E:CC}
L\geq \|\p_q^2(1/u)\|_\alpha^2+C_1^2\left(1+\sum_{l=0}^1(2\|\p_q^l (1/u)\|_\alpha +\|\p_q^{1+l}u\|_\alpha)\right)^{12},
\end{equation}
such that $\|\p_q^n u\|_\alpha\leq L^{n-2}(n-3)!$
for all $3\leq n\leq N$ and $\inf_{\0} u>0$,
then we have
\begin{equation}\label{E:4}
 \|\p_q^n(1/u)\|_\alpha\leq L^{n-3/2}(n-2)!\qquad\text{for all $2\leq n\leq N.$}
\end{equation}
\end{lemma}
\begin{proof}
 The proof follows by induction.
 By the choice of $L$, it is clear that the relation \eqref{E:4} is satisfied when $n=2$.
 So, let us assume that \eqref{E:4} is satisfied for all $2\leq n\leq m-1,$ with $3\leq m\leq N.$
 In order to prove the assertion for $n=m$, we write
 \[
 \p_q^m(1/u)=\p_q^{m-1}(u_1u_2u_3)
 \]
 whereby $u_1:=-\p_qu$ and $u_2=u_3:=1/u.$
 Our hypothesis and the induction assumption ensure that for all $2 \leq n\leq m-1$ we have 
 \begin{align*}
 &\|\p_q^n u_1\|_\alpha=\|\p_q^{n+1}u \|_\alpha\leq L^{n-1}(n-2)!,\\
 &\|\p_q^nu_2\|_\alpha=\|\p_q^{n}(1/u)\|_\alpha\leq L^{n-3/2}(n-2)!\leq L^{n-1}(n-2)!,
 \end{align*}
 the last inequality being a consequence of the fact that $L>1.$
 Whence, Lemma  \ref{L:1} $(ii)$ and  the relation \eqref{E:CC} combine to
 \begin{align*}
  \|\p_q^m(1/u)\|_\alpha&\leq C_1\left(1+\sum_{l=0}^1(2\|\p_q^l (1/u)\|_\alpha +\|\p_q^{1+l}u\|_\alpha)\right)^6L^{m-2}(m-3)!\\
  &\leq L^{m-3/2}(m-2)!,
 \end{align*}
which completes the proof.
\end{proof}

We come now to the proof of our main result. 

\begin{proof}[Proof of Theorem \ref{MT}] Let $h$ be a weak solution of \eqref{H}-\eqref{CH} and let $L$ be a positive constant such that
\begin{align}
L\geq& \|\p_q^2(1/h_p)\|_\alpha^2+\|\p_q^2h_q\|_\alpha^2+C_1^2\left(1+\sum_{l=0}^1(2\|\p_q^l (1/h_p)\|_\alpha +\|\p_q^{1+l}h_p\|_\alpha)\right)^{12}+\sum_{l=0}^4\|\p_q^l h\|_{1+\alpha}.\label{E:C}
\end{align}
Then, it is clear that $L\geq1.$
Moreover, the inequality \eqref{E:C} guarantees  that \eqref{E} is satisfied for $m=3$ and $m=4$.
So, let us presuppose that \eqref{E} is true for all $3\leq n\leq m-1,$ whereby $m\geq5.$
We need to show only  that \eqref{E} is satisfied for $m.$
To this end, let us observe that
\begin{align}\label{S1}
 \max\{\|\p_q^nh_q\|_\alpha,\|\p_q^nh_p\|_\alpha\}\leq \|\p_q^nh\|_{1+\alpha}\leq L^{n-2}(n-3)! \qquad\text{for $3\leq n\leq m-1.$}
\end{align}
This estimate together with the Lemma \ref{L:2}, which we may apply to the function $u=h_p$, cf. \eqref{CH} and \eqref{E:C}, yield
 \begin{align}\label{S2}
 \|\p_q^n(1/h_p)\|_\alpha\leq  L^{n-3/2}(n-2)! \qquad\text{for all $2\leq n\leq m-1.$}
\end{align}
With our choice of $L$ and in view of the induction assumption, we also have that
\begin{align}\label{S3}
 \|\p_q^nh_q\|_\alpha\leq L^{n-3/2}(n-2)! \qquad\text{for $2\leq n\leq m-1.$}
\end{align}
 
Recall that $\p_q^mh$ is the solution of the elliptic problem \eqref{WH2}. 
The arguments used in  the proof of Proposition \ref{P:1}  and the Schauder  estimate derived in  \cite[Theorem 3]{CS11}  ensure the existence of a positive constant $C_3 $ such that
\begin{equation}\label{SE}
 \|\p_q^mh\|_{1+\alpha}\leq C_3\left(\|\p_q^m h\|_0+\|f_m\|_\alpha+\|g_m\|_\alpha+\|\varphi_m\|_\alpha\right),
\end{equation}
meaning that we are left to prove that the right-hand side of relation  \eqref{SE} may be estimated by $L^{m-2}(m-3)!.$

The supremum norm $\|\p_q^m h\|_0$ can be bounded by using the induction assumption only
\begin{align}\label{F1}
\|\p_q^m h\|_0\leq\|\p_q^{m-1} h\|_\alpha\leq L^{m-3}(m-4)!.
\end{align}

The terms appearing in the definition of $f_m$ and $g_m$ can be estimated by using the same scheme. 
Indeed, let us notice that the estimate \eqref{E:1} of Lemma \ref{L:1} together with the estimates \eqref{S2} and \eqref{S3}  ensure the existence of a constant $K_0>1$ with the property that
\begin{align}\label{G:1}
 \max\left\{\left\|\p_q^n\left(\frac{1}{h_p}\right)\right\|_\alpha,
\left\| \p_q^n\left(\frac{h_q}{h_p^2}\right)\right\|_\alpha, \left\|\p_q^n\left(\frac{1+h^2_q}{h_p^3}\right)\right\|_\alpha\right\}\leq K_0L^{n-3/2}(n-2)!
\end{align}
for all $2\leq n\leq m-1.$  
With this preparation, we write the first sum appearing in the definition of $f_m$ as follows 
\begin{align*}
 \sum_{n=1}^{m-1}\begin{pmatrix}m-1\\n\end{pmatrix}\p_q^n\left(\frac{1}{h_p}\right)\p_q(\p_q^{m-n}h)=\left(\sum_{n=1}^1+\sum_{n=2}^{m-3}+\sum_{n=m-2}^{m-1}\right)\begin{pmatrix}m-1\\n\end{pmatrix}\p_q^n\left(\frac{1}{h_p}\right)\p_q(\p_q^{m-n}h), 
\end{align*}
and observe that the convergence of the series $\sum_{n}n^{-2}$ implies
\begin{align*}
 \left\|\sum_{n=2}^{m-3}\begin{pmatrix}m-1\\n\end{pmatrix}\p_q^n\left(\frac{1}{h_p}\right)\p_q(\p_q^{m-n}h)\right\|_\alpha&\leq \sum_{n=2}^{m-3}\begin{pmatrix}m-1\\n\end{pmatrix}\left\|\p_q^n\left(\frac{1}{h_p}\right)\right\|_\alpha\|\p_q^{m-n}h\|_{1+\alpha}\\
 &\leq K_0 L^{m-7/2}\sum_{n=2}^{m-3}\begin{pmatrix}m-1\\n\end{pmatrix}(n-2)!(m-n-3)!\\
 &\leq K_0 L^{m-7/2}(m-3)!\sum_{n=2}^{m-3}\frac{(m-1)^2}{(n-1)^2(m-n-2)^2}\\
 &\leq K_0L^{m-7/2}(m-3)!.  
\end{align*}
On the other hand, it follows readily from \eqref{S2} and the induction assumption that
\begin{align*}
 \left\|\left(\sum_{n=1}^1 +\sum_{n=m-2}^{m-1}\right)\begin{pmatrix}m-1\\n\end{pmatrix}\p_q^n\left(\frac{1}{h_p}\right)\p_q(\p_q^{m-n}h)\right\|_\alpha\leq K_1L^{m-5/2}(m-3)!.
\end{align*}
Since the remaining sums  that appear in the definition of $f_m$ and $g_m$ can be estimated in a similar way, we conclude that
\begin{align}\label{F2}
 \|f_m\|_\alpha+\|g_m\|_\alpha\leq K_2L^{m-5/2}(m-3)!.
\end{align}

It remains to estimate the norm $\|\varphi_m\|_\alpha$.
Because $C^{\alpha}(\ov\0)$ is an algebra, we need to estimate only the terms
\[
T_1:=\sum_{n=1}^{m-1} \begin{pmatrix}m\\n\end{pmatrix}(\p_q^nh_p)\p_q^{m-n}h_p\qquad\text{and}\qquad T_2:=\sum_{n=1}^m\begin{pmatrix}m\\n\end{pmatrix}(\p_q^nh)\p_q^{m-n}h_p^2.
\]
In order to deal with $T_1$, we write
\[
T_1=\left(\sum_{n=1}^2+\sum_{n=3}^{m-3}+\sum_{n=m-2}^{m-1}\right)\begin{pmatrix}m\\n\end{pmatrix}(\p_q^nh_p)\p_q^{m-n}h_p,
\]
and obtain from the induction assumption that
\begin{align*}
 \left\|\left(\sum_{n=1}^2 +\sum_{n=m-2}^{m-1}\right)\begin{pmatrix}m\\n\end{pmatrix}(\p_q^nh_p)\p_q^{m-n}h_p\right\|_\alpha\leq &K_3L^{m-3}(m-3)!.
\end{align*}
On the other hand, if $m\geq 6,$   the induction assumption  together with   the convergence of the series   $\sum_{n}n^{-3}$  imply  
\begin{align*}
 \left\|\sum_{n=3}^{m-3}\begin{pmatrix}m\\n\end{pmatrix}(\p_q^n h_p)\p_q^{m-n}h_p\right\|_\alpha&\leq \sum_{n=3}^{m-3}\begin{pmatrix}m\\n\end{pmatrix}\left\|\p_q^nh\right\|_{1+\alpha}\|\p_q^{m-n}h\|_{1+\alpha}\\
 &\leq \sum_{n=3}^{m-3}\begin{pmatrix}m\\n\end{pmatrix}L^{n-2}(n-3)!L^{m-n-2}(m-n-3)!\\
 &\leq   L^{m-4}(m-3)!\sum_{n=3}^{m-3}\frac{m^3}{(n-2)^3(m-n-2)^3}\\
 &\leq  C_4 L^{m-4}(m-3)!,
\end{align*}
and we conclude that
\begin{align}\label{F3}
\|T_1\|_\alpha\leq K_4 L^{m-3}(m-3)!.
\end{align}

Finally, in order to estimate $T_2,$ we write
\[
T_2=\left(\sum_{n=1}^2+\sum_{n=3}^{m-3}+\sum_{n=m-2}^{m}\right) \begin{pmatrix}m\\n\end{pmatrix}(\p_q^nh)\p_q^{m-n}h_p^2
\]
and infer from the relation \eqref{S1}  and Lemma \ref{L:1} $(iii)$ that there  exists  a constant $K_5>1$ such that $\|\p_q^nh_p^2\|_\alpha\leq K_5L^{n-2}(n-3)!$ for all $3\leq n\leq m-1.$
This estimate combined with the induction assumption guarantee the existence of a constant $K_6>1$ with the property that
\begin{align*}
 \left\|\left(\sum_{n=1}^2+\sum_{n=m-2}^{m}\right) \begin{pmatrix}m\\n\end{pmatrix}(\p_q^nh)\p_q^{m-n}h_p^2\right\|_\alpha&\leq K_6 L^{m-3}(m-3)!
\end{align*}
and, when $m\geq 6,$ we use again the convergence of the series   $\sum_{n}n^{-3}$ to find
\begin{align*}
 \left\|\sum_{n=3}^{m-3} \begin{pmatrix}m\\n\end{pmatrix}(\p_q^nh)\p_q^{m-n}h_p^2\right\|_\alpha&\leq C_5\sum_{n=3}^{m-3} \begin{pmatrix}m\\n\end{pmatrix}\|\p_q^nh\|_{1+\alpha}\|\p_q^{m-n}h_p^2\|_\alpha\\
 &\leq K_5\sum_{n=3}^{m-3} \begin{pmatrix}m\\n\end{pmatrix}L^{n-2}(n-3)!L^{m-n-2}(m-n-3)!\\
 & \leq K_5L^{m-4}(m-3)!.
 \end{align*}
Thus, we have found a constant $K_7>1$ such that
\begin{align}\label{F4}
\|T_2\|_\alpha\leq K_7 L^{m-3}(m-3)!.
\end{align}

Combining \eqref{SE}, \eqref{F1}, \eqref{F2}, \eqref{F3}, and \eqref{F4} we conclude that there exists a constant $K_8>1$ such  that $\|\p_q^mh\|_{1+\alpha}\leq K_8 L^{m-5/2}(m-3)!.$
Since $K_8$ is independent of $m$ and $L$, we may require, additionally to \eqref{E:CC},  that $L\geq K^2_8.$
But then $\|\p_q^mh\|_{1+\alpha}\leq   L^{m-2}(m-3)!,$  and the induction principle guarantees that\eqref{E} is satisfied for all $m\in\N$ with $m\geq3.$ 
This finishes the proof. 
\end{proof}

\end{document}